\documentclass[12pt,a4paper]{article}

\usepackage{titlesec}
\usepackage{tocloft}
\usepackage{authblk}
\usepackage{amsmath,amssymb,amscd,amsthm,latexsym,accents,stmaryrd}
\usepackage{mathrsfs}
\usepackage{mathtools}
\usepackage{amsfonts}
\usepackage{bbm}
\usepackage[all,cmtip]{xy}
\usepackage{faktor}
\usepackage{enumitem}
\usepackage{url}
\usepackage{hyperref}
\usepackage{tikz}
\usepackage{tikz-cd}
\usetikzlibrary{matrix,arrows,backgrounds}
\usetikzlibrary{cd}
\usetikzlibrary{decorations.text}
\usetikzlibrary{patterns}
\usepackage{centernot}

\usetikzlibrary{automata,positioning}
\usepackage[parfill]{parskip}

\makeatletter
\def\thanks#1{\protected@xdef\@thanks{\@thanks
		\protect\footnotetext{#1}}}
\makeatother

\makeatletter
\newcommand{\subjclass}[2][2020]{%
	\let\@oldtitle\@title%
	\gdef\@title{\@oldtitle\footnotetext{#1 \emph{Mathematics subject classification:} #2}}%
}
\newcommand{\keywords}[1]{%
	\let\@@oldtitle\@title%
	\gdef\@title{\@@oldtitle\footnotetext{\emph{Key words and phrases:} #1.}}%
}
\makeatother

\makeatletter
\DeclareFontEncoding{LS2}{}{\@noaccents}
\makeatother
\DeclareFontSubstitution{LS2}{stix}{m}{n}
\DeclareSymbolFont{largesymbolsstix}{LS2}{stixex}{m}{n}
\DeclareMathDelimiter{\lbrbrak}{\mathopen}{largesymbolsstix}{"EE}{largesymbolsstix}{"14}
\DeclareMathDelimiter{\rbrbrak}{\mathclose}{largesymbolsstix}{"EF}{largesymbolsstix}{"15}

\newcommand{\nocontentsline}[3]{}
\newcommand{\tocless}[2]{\bgroup\let\addcontentsline=\nocontentsline#1{#2}\egroup}

\theoremstyle{definition}
\newtheorem{defin}{Definition}[section]
\newtheorem{rem}[defin]{Remark}
\newtheorem{rems}[defin]{Remarks}
\newtheorem{no}[defin]{Notation}
\newtheorem{ex}[defin]{Example}
\newtheorem{exs}[defin]{Examples}
\theoremstyle{plain}
\newtheorem{theor}[defin]{Theorem}
\newtheorem{lem}[defin]{Lemma}
\newtheorem{prop}[defin]{Proposition}
\newtheorem{cor}[defin]{Corollary}
\newtheoremstyle{dotless-thm}
{3pt}
{3pt}
{}
{}
{}
{.}
{.5em}
{}
\theoremstyle{dotless-thm}

\def\Z{{\mathbb{Z}}}
\def\im{{\mbox{im}}\hspace{0.04cm}}

\def\R{{\mathbb{R}}}
\def\K{{\mathbb{K}}}
\def\I{{\mathcal{I}}}
\allowdisplaybreaks

\author[1]{Catarina Faustino}
\author[1]{Thomas Kahl}
\thanks{This research was supported by FCT (\emph{Funda\c c\~ao para a Ci\^encia e a Tecnologia}, Por\nolinebreak tugal) through the first author's PhD scholarship UI/BD/152071/2021 and projects UIDB/00013/2020 and UIDP/00013/2020.} 

\affil{\small{Centro de Matem\'atica,
		Universidade do Minho, 
		4710-057 Braga,
		Portugal\\
		\texttt{catarina.0109.cf@gmail.com}, 
		\texttt{kahl@math.uminho.pt}
	}
}

\begin{document}
	
	\title{The homology digraph of a preordered space}
	
	\date{}

	

	\maketitle
	
	\begin{abstract}
			This paper studies a notion of directed homology for preordered spaces, called the homology digraph. We show that the homology digraph is a directed homotopy invariant and establish variants of the main results of ordinary singular homology theory for the homology digraph. In particular, we prove a K\"unneth formula, which  enables one to compute the homology digraph of a product of  preordered spaces from the homology digraphs of the components.
	\end{abstract}

\begin{sloppypar}
	
\section{Introduction}
	
This paper is situated in the field of directed algebraic topology, which studies spaces equipped with a supplementary structure representing the flow of time or a direction of traversal. There exist a variety of frameworks for directed topology, such as partially ordered spaces (pospaces) \cite{FajstrupGR, Goubault, reldi}, d$\mbox{-}$spaces \cite{GrandisDirHoTI, GrandisBook}, streams \cite{KrishnanConvCat, KrishnanNorth}, and (pre)cubical sets \cite{FajstrupGR, GrandisBook}. Our goal in this work is to develop a theory of directed homology for preordered spaces \cite{GrandisBook}. The aforementioned kinds of directed spaces can be naturally interpreted as preordered spaces, and so this theory also applies in these other settings. 


Directed algebraic topology has applications in concurrency theory, the domain of theoretical computer science that deals with systems of simultaneously executing processes. The state space of a concurrent system can be modeled as a directed space, and the executions of the system are then represented by directed paths, i.e., paths that evolve in accordance with the given direction. It turns out that two such execution paths can be considered equivalent from a computer science point of view if and only if they are homotopic relative to the endpoints through directed paths. This link between topology and concurrency theory is actually  the origin of the field of directed algebraic topology. Further connections between directed algebraic topology and concurrency theory are developed in \cite{FGHMR, FajstrupGR, Goubault}.

An important line of research in directed algebraic topology is directed homology. A number of different notions of directed homology have been introduced in the literature, see, e.g., \cite{DubutGG, FahrenbergDiH, GoubaultJensen, GrandisBook, hgraph}. A natural approach to directed homology is to consider ordinary homology with a supplementary directional structure. Probably the best-known example of such a concept of directed homology is the one defined by Grandis for cubical sets, where the extra structure is a degreewise additive preorder \cite{GrandisDiH, GrandisBook}. Unfortunately, it turns out that this preorder is always discrete (in positive degrees) for cubical subsets of subdivided cubes, which correspond to Euclidean pospaces, i.e., subspaces of \(\R^n\) equipped with the natural partial order. A notion of directed homology that distinguishes between homotopy equivalent Euclidean pospaces with different directed structures is the homology graph \cite{hgraph}, which can be defined for various kinds of directed spaces. This is a directed graph with vertices the homology classes and edges representing a directional relation between them. For example, in the homology graph of the Euclidean pospace on the left hand side of Figure \ref{figeuclid}, there exists an edge from the homology
class representing the lower hole to the homology class representing the upper hole. In contrast, there are no edges between nontrivial homology classes of degree 1 in the homology graph of the Euclidean pospace on the right hand side of Figure \ref{figeuclid}.


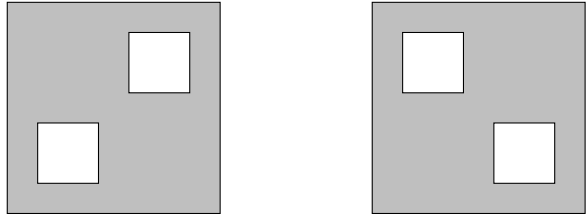
\begin{figure} \label{figeuclid}
	\begin{center}
		\begin{tikzpicture}[scale=0.4,on grid]
			\path[draw, fill=lightgray] (0,0)--(7,0)--(7,7)--(0,7)--cycle;
			\path[draw, fill=white]
			(1,1)--(3,1)--(3,3)--(1,3)--cycle
			(4,4)--(6,4)--(6,6)--(4,6)--cycle;
			\path[draw, fill=lightgray] (12,0)--(19,0)--(19,7)--(12,7)--cycle;
			\path[draw, fill=white]
			(13,4)--(15,4)--(15,6)--(13,6)--cycle
			(16,1)--(18,1)--(18,3)--(16,3)--cycle;		
		\end{tikzpicture}
		\caption{Homotopy equivalent Euclidean pospaces with different directed structures}
	\end{center}
\end{figure}


The product of directed spaces models the parallel composition of independent concurrent systems and is hence a construction of fundamental importance from both a topological and a computer science point of view. One can therefore argue that a directed homology theory should provide a Künneth formula that permits one to compute the directed homology of a product of directed spaces from the directed homology of the components. 
Unfortunately, the topological Künneth theorem does not extend in a satisfying way to either Grandis' directed homology (see \cite{CatarinaMestrado}) or the homology graph. In the case of the homology graph, this is due to a lack of compatibility between the edge relation and the linear structure of homology.

In this paper, we introduce the homology digraph of a preordered space, which is a modified homology graph where the edge relation and the linear structure are better integrated with each other. More precisely, we work over a field and introduce directional graded vector spaces, which are graded vector spaces equipped with what we call a bilinear relation. The homology digraph of a preordered space is then defined to be the directional graded vector space obtained by equipping the homology of the space with the bilinear relation generated by the edge relation of the usual homology graph. We show that the homology digraph is a directed homotopy invariant and demonstrate its compatibility with the main ingredients of ordinary singular homology theory. In particular, we extend the Künneth theorem to the homology digraph and prove, more precisely, that the cross product homomorphism is an isomorphism of directional graded vector spaces.


\section{Preordered spaces}

The purpose of this section is to present some basic material on preordered spaces and to indicate how they relate to other kinds of directed spaces.

\paragraph{The category of preordered spaces.} A \emph{preordered space} is a topological space \(X\) with  a preorder relation \(\preceq_X\) on it \cite{GrandisBook}. Preordered spaces form a category, in which morphisms are monotone (i.e., preorder-preserving) continuous maps. This category is complete and cocomplete, with limits and colimits created in the category of topological spaces. In particular, the topological product and the topological coproduct of a family \((X_i)_{i\in \I}\) of preordered spaces are preordered spaces. The preorder of the product \(\prod_{i\in \I} X_i\) is the componentwise preorder, i.e., $$ (x_i)\preceq_{\prod_{i\in \I} X_i} (y_i) \iff \forall \,i\in \I\quad x_i \preceq_{X_i} y_i.$$ 
The preorder of the coproduct is defined as follows:
$$x \preceq_{\coprod_{i\in \I} X_i} y \iff \exists \, j \in \I\quad x, y \in X_j, \; x \preceq_{X_j} y.$$
A preordered space \(A\) is called a \emph{preordered subspace} of a preordered space \(X\) if \(A\) is a  subspace of \(X\) and the relation \(\preceq_A\) satisfies
\[\forall \, a,b \in A\quad a \preceq_A b \iff a \preceq_X b.\]
The forgetful functor from the category of preordered spaces to the category of topological spaces has both a left and a right adjoint. The left adjoint associates with a topological space the same space with the \emph{discrete} preorder, i.e., the preorder where every element is only related to itself. The right adjoint equips a space with the \emph{indiscrete} preorder, where any two elements are related.

\paragraph{Dihomotopy.} 
Let \(I\) denote the unit interval with the discrete preorder. Two monotone maps \(f, g\colon X \to Y\) are called \emph{dihomotopic} if there exists a \emph{dihomotopy} between them, i.e., a monotone map \(H\colon X \times I \to Y\) such that, for all \(x \in X\), \(H(x,0) = f(x)\) and \(H(x,1) = g(x)\). Thus, two monotone maps are dihomotopic if and only if they are homotopic through monotone maps. A monotone map \(f\colon X \to Y\) is called a \emph{dihomotopy equivalence} if it admits a \emph{dihomotopy inverse}, meaning a monotone map \(g\colon Y \to X\) such that the composites \(g\circ f\) and \(f\circ g\) are dihomotopic to the identities of \(X\) and \(Y\), respectively.

\paragraph{Pospaces.} 
A \emph{partially ordered space} or \emph{pospace} is a topological space \(X\) equipped with a partial order, which some authors require to be closed as a subspace of \(X \times X\), see, e.g., \cite{FajstrupGR}. Thus, by definition, pospaces are preordered spaces. A standard example is the ordered circle $O^1$, which is the circle $S^1$ with both the left and the right semicircles totally ordered from the bottom to the top, see Figure \ref{figO1}.
 		\begin{figure} \label{figO1}
 			\centering
 			\begin{tikzpicture}
 			\draw[thick,->,>=stealth] (0,-2) arc (-90:86:1) node[midway, right]{\scalebox{0.75}{}};
 			\draw[thick,->,>=stealth] (0,-2) arc (-90:-266:1) node[midway, left]{\scalebox{0.75}{}};
 			
 			\node[state,minimum size=0pt,inner sep =1.5pt,fill=white] (p_0) at (0,0)  {};
 			\node[state,minimum size=0pt,inner sep =1.5pt,fill=white] (p_1) at (0,-2)  {};
 			
 			\path
 			(p_1) node[below]{\scalebox{0.95}{}}
 			(p_0) node[above]{\scalebox{0.95}{}};
 			\end{tikzpicture}
 			\caption{Ordered circle $O^1$}
 		\end{figure}
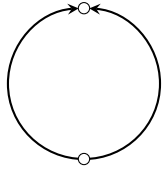
   A class of examples that are frequently considered are the \emph{Euclidean pospaces}, i.e., subspaces of \(\R^n\) with the natural partial order. Two such Euclidean pospaces are shown in Figure \ref{figeuclid}. 

\paragraph{D-spaces.} A \emph{d-space} is a topological space \(X\) together with a set \(dX\) of \emph{directed paths} that contains all constant paths and is closed under concatenation and (not necessarily strictly) increasing reparametrization \cite{GrandisBook}. A d-space \(X\) is naturally a preordered space with respect to the \emph{reachability preorder} given by $$x\preceq_X y \iff \exists\, \alpha \in dX \quad \alpha (0)=x,\, \alpha(1)=y.$$ 
The category of d-spaces permits one to distinguish the \emph{directed circle} \(\vec S^1\), where the directed paths are the counterclockwise oriented paths, from the \emph{natural circle}, where all paths are directed. However, in both cases, the reachability preorder is the indiscrete preorder, and so as preordered spaces, the two circles are the same. We will refer to the circle with the indiscrete preorder as the directed circle.


\paragraph{Streams.} A \emph{stream} is a topological space equipped with a \emph{circulation}, meaning a function assigning to each open set a preorder such that the preorder of a union of open sets is generated by the preorders of these open sets \cite{KrishnanConvCat}. In particular, considering the preorder assigned to itself, a stream is a preordered space.


\paragraph{Precubical and cubical sets.} A \emph{precubical set} is a graded set $X = (X_n)_{n\geq 0}$ with face maps $d_i^k \colon X_n \to X_{n-1}$ $(i\in \{1,\dots,n \}, k\in \{ 0,1\})$ that satisfy the compatibility condition $d_i^k d_j^l = d_{j-1}^l d_i^k$ $(i<j)$. 
A \emph{cubical set} is a precubical set $X$ with degeneracy maps $e_i\colon X_{n-1} \to X_n$ $(n>0,\, i \in \{1, \dots, n\})$ 
satisfying the additional relations $e_je_i = e_ie_{j-1}$ $(i<j)$
and
\[d^k_ie_j = \left\{\begin{array}{ll}
	e_{j-1}d^k_i &\;(i<j),\\
	id &\;(i=j),\\
	e_jd^k_{i-1} &\;(i>j).	
\end{array}
\right.\]
Both precubical and cubical sets can be realized geometrically as d-spaces and hence as preordered spaces, see, e.g., \cite{FGHMR, GrandisBook}. 


\section{Directional graded vector spaces} 

By definition, the homology digraph of a preordered space will be a directional graded vector space, i.e., a graded vector space equipped with a bilinear relation. In this section, we define these algebraic concepts and prove fundamental results about them. Throughout this paper we work over a field  \(\mathbb K\). Given a graded vector space \(V = (V_k)_{k\in \Z}\), we say that an element \(v\in V_k\) is an \emph{element of \(V\) of degree \(k\)} and write \(v\in V\) and \(\deg(v) = k\). 

\paragraph{Bilinear relations.} A \emph{relation} on a graded vector space \(V = (V_k)_{k\in \Z}\) is a relation on the disjoint union \(\coprod _{k \in \Z} V_k\). A relation \(\searrow\) on a graded vector space \(V\) is said to be \emph{bilinear} if there exists a graded vector subspace \(R \subseteq V \otimes V\) such that
\[\forall  v, w \in V\quad  v \searrow w \iff v\otimes w \in R.\]
Such a graded vector subspace will be called a \emph{defining vector space} for \(\searrow\). The smallest defining vector space for \(\searrow\), i.e., the intersection of all defining vector spaces for \(\searrow\), will be denoted by \(R^\searrow\). The following properties follow immediately from the definition:

\begin{prop} \label{bilin}
	Let \(\searrow\) be a bilinear relation on a graded vector space \(V\). Then for all \(v,v', w, w'\in V\) and \(\lambda, \mu \in \mathbb{K}\),
	\begin{enumerate}
		\item \(v = 0\) or \(w=0 \implies v \searrow w\);
		\item \(v \searrow w \implies \lambda v \searrow \mu w\);
		\item \(\deg(v) = \deg(v'), v\searrow w, v'\searrow w \implies v + v' \searrow w\);
		\item \(\deg(w) = \deg(w'), v\searrow w, v\searrow w' \implies v \searrow  w + w'\).
	\end{enumerate}
\end{prop}

\begin{ex}
	Perhaps surprisingly, a relation that satisfies these properties does not necessarily have to be bilinear. Indeed, suppose that $\mathbb{K} = \Z_2$, and let $V$ be the graded vector space concentrated in degree $0$ defined by $V_0 = \Z_2^3$. Write $e_1 = (1,0,0)$, $e_2=(0,1,0)$, and $e_3 = (0,0,1)$. Define the relation $\searrow$ by
	\begin{enumerate}[label=(\roman*)]
		\item $v \searrow w$  for all  $v,w\in V$ such that \(v= 0\) or \(w=0\);
		\item $e_1\searrow e_2+e_3,\;\; e_2 \searrow e_2,\;\; e_3 \searrow e_1+e_2,\;\; e_1+e_2\searrow e_1,\;\; e_2+e_3 \searrow e_3$.		
	\end{enumerate}
	Then, by (i), $\searrow$ has properties 1 and 2 of Proposition \ref{bilin}. In order to check property 3, let $v,v',w\in V$ such that \(\deg(v) = \deg(v')\), $v \searrow w$, and $v' \searrow w$. We have to show that $v+v' \searrow w$. If $v=v'$, then this holds by condition (i). So we may suppose that $v$ and $v'$ are different and, in particular, that at least one of them is nonzero. By condition (i), we may suppose that $w \not= 0$. Then $w$ is one of the elements $e_2+e_3$, $e_2$, $e_1+e_2$, $e_1$, and $e_3$. In each case, there exists precisely one nonzero element $x$ such that $x \searrow w$. Hence one of the elements $v$ and $v'$ is this $x$ and the other is $0$. Thus  $v+v' = x \searrow w$. An analogous argument shows that $\searrow$ has property 4 of Proposition \ref{bilin}. Now consider the element $v = e_1+e_2+e_3$. Since 
	\[v\otimes v = e_1\otimes (e_2+e_3) + e_2 \otimes e_2 +  e_3 \otimes (e_1+e_2) + (e_1+e_2)\otimes e_1  +  (e_2+e_3) \otimes e_3,\]
	by (ii), we would have $v\searrow v$ if $\searrow$ was bilinear. Since, however, $v\not \searrow v$, we conclude that $\searrow$ is not bilinear.
\end{ex}

We omit the easy proof of the following fact:

\begin{prop} \label{intersect}
	 The intersection of a nonempty family of bilinear relations on a graded vector space is bilinear.
\end{prop}


Let \(\nearrow\) be an arbitrary relation on a graded vector space \(V\). The smallest bilinear relation \(\searrow\) satisfying \(\nearrow \subseteq \searrow\), which exists by Proposition \ref{intersect},  
will be called the \emph{bilinear relation generated by}  \(\nearrow\). 

\begin{prop} \label{1.3}
	Let \(\searrow\) be the bilinear relation generated by a relation \(\nearrow\) on a graded vector space \(V\). Then 
	\(R^\searrow = \langle x\otimes y \mid  x \nearrow y \rangle\).
\end{prop}

\begin{proof}
	Write \(R = \langle x\otimes y \mid  x \nearrow y \rangle\). Since \(\nearrow \subseteq \searrow\), we have \(R \subseteq R^\searrow\). For the reverse inclusion, it is enough to show that \(\searrow\) coincides with the bilinear relation \(\uparrow\) defined by \(R\). Since \(\nearrow \subseteq \uparrow\), we have \(\searrow \subseteq \uparrow\). Consider elements \(v, w \in V\) such that \(v \uparrow w\). Then 
	\[v \otimes w = \sum_r \lambda_r x_r\otimes y_r, \quad \lambda_r\in \K,\; x_r \nearrow y_r.\]
	Since \(x_r \searrow y_r\), we have \(x_r\otimes y_r \in R^\searrow\) and hence \(v\otimes w\in R^\searrow\). Thus, \(v\searrow w\). 
\end{proof}

\paragraph{Directional graded vector spaces.} A \emph{directional graded vector space} is a graded vector space \(V\) equipped with a bilinear relation \(\searrow_V\). We refer to  \(\searrow_V\) as the \emph{pointing relation} of \(V\), and if \(v \searrow_V w\), we say that \(v\) \emph{points to} \(w\). Directional graded vector spaces form a category, in which morphisms are morphisms of graded vector spaces that are compatible with the pointing relations. The following lemma is useful for checking that a map is a morphism of directional graded vector spaces:

\begin{lem} \label{1.5}
	Let \(V\) and \(W\) be directional graded vector spaces, and suppose that the pointing relation \(\searrow_V\) is generated by a relation \(\nearrow_V\). Let \(f \colon V \to W\) be a morphism of graded vector spaces such that 
	\[\forall\,  v, v' \in V \quad v \nearrow_V v' \implies f(v) \searrow_W f(v').\]
	Then \(f\) is a morphism of directional graded vector spaces. 
\end{lem}

\begin{proof}
	Let $v, v'\in V$ be such that $v\searrow_V v'$. Then, by Proposition \ref{1.3}, $$v\otimes v' = \sum_r \lambda_r x_r \otimes y_r, \quad \lambda_r \in \K, \;  x_r \nearrow_V y_r.$$ By our hypothesis, $f(x_r) \searrow_W f(y_r)$, i.e., $f(x_r)\otimes f(y_r) \in R^{\searrow_W}$. Hence
	\begin{align*}
		f(v) \otimes f(v')  
		&=  \sum_r \lambda_r f(x_r) \otimes f(y_r) \in R^{\searrow_W}
	\end{align*}	
	and therefore $f(v)\searrow_W f(v')$. 		  
\end{proof}

\paragraph{Direct sum.}
	The direct sum of a family \((V^i)_{i \in \I}\) of directional graded vector spaces is a directional graded vector space, with  pointing relation \(\searrow_{\bigoplus _{i\in \I} V^i}\) defined by the graded vector space \[\bigoplus _{i \in \I} R^{\searrow_{V^i}} \subseteq \bigoplus _{i \in \I} V^i\otimes V^i \subseteq \bigoplus _{i \in \I} V^i\otimes \bigoplus _{i \in \I} V^i.\]


\begin{prop} \label{oplus}
	Let \((V^i)_{i\in \I}\) be a family of directional graded vector spaces, and let \({\alpha, \beta \in \bigoplus_{i\in \I} V^i}\) be nonzero elements. Then \({\alpha \searrow_{\bigoplus_{i\in \I} V^i} \beta}\) if and only if there exists an index \(j\in \I\) such that \(\alpha, \beta \in V^j\) and \(\alpha \searrow_{V^j} \beta\). 	
\end{prop}

\begin{proof}
	If there exists an index \(j\in \I\) such that \(\alpha, \beta  \in V^j\) and \(\alpha \searrow_{V^j} \beta\), then \(\alpha \otimes \beta \in R^{\searrow_{V^j}} \subseteq \bigoplus _{i \in \I} R^{\searrow_{V^i}}\) and therefore \({\alpha \searrow_{\bigoplus_{i\in \I} V^i} \beta}\).
	
	Suppose now that \({\alpha \searrow_{\bigoplus_{i\in \I} V^i} \beta}\). Let \(v^{i_r} \in V^{i_r}\) \((r=1, \dots, m)\) and \(w^{j_s} \in V^{j_s}\) \((s = 1, \dots ,n)\) be the unique nonzero elements such that \(\alpha = \sum_r v^{i_r}\) and \(\beta = \sum_s w^{j_s}\). Then 
	\[\alpha \otimes \beta = \sum_{r,s} v^{i_r}\otimes w^{j_s} \in \bigoplus _{i \in \I} R^{\searrow_{V^i}} \subseteq \bigoplus _{i \in \I} V^i\otimes V^i.\]
	Since we are working over a field, \(v^{i_r}\otimes w^{j_s} \not= 0\) for all \(r\) and \(s\). It follows that \(m = n = 1\) and that \(i_1 = j_1\). Moreover, \(\alpha \otimes \beta = v^{i_1}\otimes w^{i_1} \in R^{\searrow_{V^{i_1}}}\). Thus, 
	\(\alpha \searrow_{V^{i_1}} \beta\).
\end{proof}

As an immediate consequence, we have the following fact:

\begin{cor} \label{coprod}
	The coproduct in the category of directional graded vector spaces is given by the direct sum.
\end{cor}

\paragraph{Tensor product.} The tensor product of two directional graded vector spaces \(V\) and \(W\) is a directional graded vector space, with pointing relation 
\(\searrow_{V\otimes W}\) defined by the graded vector subspace 
\[R = \langle v\otimes w \otimes v' \otimes w' \mid v \searrow_V v', w \searrow_W w'\rangle\]
of \(V\otimes W \otimes V \otimes W\).

\begin{prop} \label{tensor}
	Let \(V\) and \(W\) be directional graded vector spaces, with pointing relations \(\searrow_V\) and \(\searrow_W\) generated by relations \(\nearrow_V\)  and \(\nearrow_W\), respectively. Then the smallest defining vector space for the pointing relation of  \(V\otimes W\) is given by 	
	\[R^{\searrow_{V\otimes W}} = \langle v\otimes w \otimes v' \otimes w' \mid v \nearrow_V v', w \nearrow_W w'\rangle.\]
\end{prop}

\begin{proof}
Write $ K = \langle v\otimes w \otimes v' \otimes w' \mid v \nearrow_V v', w \nearrow_W w'\rangle.$ Let $v\otimes w \otimes v' \otimes w' \in K$ be a generator element. Then $v\nearrow_V v'$ and $w\nearrow_W w'$ and therefore $v\searrow_V v'$ and $w\searrow_W w'$. By definition, it follows that $v\otimes w \otimes v' \otimes w' \in R$. Thus, $v\otimes w \searrow_{V\otimes W} v'\otimes w'$ and consequently $v\otimes w \otimes v' \otimes w' \in R^{\searrow_{V\otimes W}}$. Hence $K \subseteq R^{\searrow_{V\otimes W}}$.

For the reverse inclusion, it is enough to show that $R \subseteq K$. So let ${v\otimes w \otimes v'\otimes w' \in R}$ be a generator. Then $v\searrow_V v'$ and $w\searrow_W w'$. By Proposition \ref{1.3}, we may therefore write $$v\otimes v' = \sum_i \lambda_i x_i \otimes x'_i, \quad \lambda_i\in \K, \; x_i \nearrow_V x'_i$$ and $$w\otimes w' = \sum_j \mu_j y_j \otimes y'_j, \quad \mu_j\in \K,\; y_j \nearrow_W y'_j.$$ Thus, 
\begin{align*}
	v\otimes w\otimes v' \otimes w' 
	& = \sum_{i,j} \lambda_i \mu_j x_i \otimes y_j \otimes x'_i \otimes y'_j \in K. \qedhere
\end{align*}

\end{proof}

\section{The homology digraph}

In this section, we define the homology digraph of a  preordered space. We show that the homology digraph is a directed homotopy invariant and examine its behavior with respect to the long exact homology sequence of a pair of preordered spaces. Furthermore, we establish an excision theorem for the homology digraph and show that the homology digraph is compatible with coproducts.

\begin{defin}
	Let \(X\) be a preordered space, and let \(A\) be a preordered subspace of \(X\). The \emph{homology digraph} of the pair \((X,A)\) is the directional graded vector space \(H_*(X,A)\)  where the pointing relation is generated by the relation \(\nearrow\) defined as follows: given relative homology classes \(\alpha, \beta  \in H_*(X,A)\) (of possibly different degrees), we set \(\alpha \nearrow \beta\) if there exist subspaces \(E, F \subseteq X\) such that 
	\begin{enumerate}[label=(\roman*)]
		\item \(\alpha \in \im H_*((E, E\cap A) \hookrightarrow (X,A))\);
		\item \(\beta \in \im H_*((F, F\cap A) \hookrightarrow (X,A))\); 
		\item \(\forall\, x\in E, y \in F\quad x \preceq_X y\).
	\end{enumerate}
	The \emph{homology digraph} of \(X\) is defined to be the homology digraph of the pair \((X,\emptyset)\). The pointing relations of the homology digraphs \(H_*(X,A)\) and \(H_*(X)\) and their generating relations will be denoted by \(\searrow_{X,A}\), \(\searrow_X\), \(\nearrow_{X,A}\), and \(\nearrow _X\), respectively.		
\end{defin}

\begin{rem}
	The \emph{homology graph} of a preordered space \(X\) would be its homology equipped with the relation \(\nearrow_X\), cf. \cite{hgraph}.
\end{rem}

\begin{exs} \label{hdorddircirc}
	(a) Consider $(X,A)$ equipped with the indiscrete preorder. Since every element is related to every element, we have that every homology class points to every homology class. 
	 
	(b) Let $X$ be a path-connected preordered space with discrete preorder, and let $\alpha, \beta \in H_\ast (X)$ be nonzero elements. Then \[\alpha \searrow_X \beta \iff \deg(\alpha) = \deg(\beta)=0.\] Indeed, if $\deg(\alpha) = \deg(\beta) =0$, then $\alpha, \beta \in \im H_\ast(\{x\} \hookrightarrow X)$ for any $x\in X$. Since $x \preceq_X x$, $\alpha \nearrow_X \beta$ and therefore $\alpha \searrow_X \beta$. If, conversely, $\alpha \searrow_X \beta$, then Proposition \ref{1.3} implies that there exist nonzero homology classes $\gamma$ and \(\delta\) such that \(\deg(\gamma) = \deg(\alpha)\), \(\deg(\delta) = \deg(\beta)\), and \(\gamma \nearrow_X \delta\). Hence there exist subspaces \(E, F \subseteq X\) such that \(\gamma \in \im H_*(E \hookrightarrow X)\), \(\delta \in \im H_*(F\hookrightarrow X)\), and \(x \preceq _X y\) for all \(x\in E\) and \(y \in F\). Since \(\preceq_X\) is the discrete preorder, this implies that \(E\) and \(F\) are singletons. Hence \(\deg(\gamma) = \deg(\delta) = 0\) and therefore also \(\deg(\alpha) = \deg(\beta) = 0\).

    (c) Consider the ordered circle \(O^1\) (see Figure \ref{figO1}), and let $\alpha, \beta \in H_\ast (O^1)$ be nonzero homology classes. If $\deg(\alpha)=0$, then we have \(\alpha \in \im H_*(E\hookrightarrow O^1)\) and \(\beta \in \im H_*(F\hookrightarrow O^1)\) for $E$ the set consisting only of the minimum and $F=O^1$. Thus, in this case, $\alpha \nearrow_{O^1} \beta$ and therefore $\alpha \searrow_{O^1} \beta$. In a similar way, when $\deg(\beta) = 0$, we can consider $E=O^1$ and $F$ the set consisting only of the maximum, and we have that $\alpha \nearrow_{O^1} \beta$, which implies that $\alpha \searrow_{O^1} \beta$. If both $\alpha$ and $\beta$ have degree $1$, then necessarily $E=F=O^1$, but since the maximum is not less than or equal to the minimum, we do not have \(\alpha \nearrow_{O^1} \beta\). Hence $R_2^{\searrow_{O^1}} =0$, and thus $\alpha \centernot{\searrow}_{O_1} \beta$.
 
	
	(d) 
    Consider the Euclidean pospace on the left-hand side of Figure \ref{figeuclid}. Since every element of the boundary of the lower hole is related to every element of the boundary of the upper hole, we have that the homology class representing the lower hole points to the homology class representing the upper hole.
	
	(e) If we consider $X$ the Euclidean pospace on the right-hand side of Figure \ref{figeuclid}, there are no pointing relations between nontrivial one-dimensional homology classes. Indeed, let $\alpha$ and $\beta$ be nonzero homology classes of degree one, and suppose that $\alpha \in \im H_\ast (E \hookrightarrow X)$ and $\beta \in \im H_\ast (F \hookrightarrow X)$. Consider the regions $Z_1, Z_2, Z_3, Z_4$ depicted in the following figure:
 
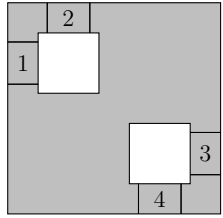
\begin{figure}[ht!] \label{figzonas}
	\begin{center}
		\begin{tikzpicture}[scale=0.4,on grid]
			\path[draw, fill=lightgray] (12,0)--(19,0)--(19,7)--(12,7)--cycle;
			\path[draw, fill=white]
			(13,4)--(15,4)--(15,6)--(13,6)--cycle
			(16,1)--(18,1)--(18,3)--(16,3)--cycle;	
             \path[draw, fill=black] (13.3,6)--(13.3,7);
             \path[draw, fill=black] (14.7,6)--(14.7,7);
             \path[draw, fill=black] (12,5.7)--(13,5.7);
             \path[draw, fill=black] (12,4.3)--(13,4.3);
             \path[draw, fill=black] (18,2.7)--(19,2.7);
             \path[draw, fill=black] (18,1.3)--(19,1.3);
             \path[draw, fill=black] (16.3,0)--(16.3,1);
             \path[draw, fill=black] (17.7,0)--(17.7,1);

             \node at (14,6.5) {\scalebox{0.75}{$2$}};
             \node at (12.5,5) {\scalebox{0.75}{$1$}};
             \node at (18.5,2) {\scalebox{0.75}{$3$}};
             \node at (17,0.5) {\scalebox{0.75}{$4$}};
		\end{tikzpicture}
	\end{center}
    
\end{figure}

Then we must have $E \cap Z_1 \neq \emptyset \neq E \cap Z_2$ or $E \cap Z_3 \neq \emptyset \neq E \cap Z_4$ and, similarly,  $F \cap Z_1 \neq \emptyset \neq F \cap Z_2$ or $F \cap Z_3 \neq \emptyset \neq F \cap Z_4$. Since \(x \not \preceq_X y\) for all \(x\in Z_2\) and \(y \in Z_1\cup Z_3 \cup Z_4\) and for all \(x\in Z_3\) and \(y \in Z_1\cup Z_2 \cup Z_4\), we do not have \(\alpha \nearrow_X \beta\). This implies that \(R^{\searrow_X}_2 = 0\), which in turn implies that \(\alpha \not \searrow_X \beta\).

\end{exs}

\paragraph{Monotone maps.} By the following proposition, the homology digraph is functorial:

\begin{prop} \label{nat}
	Let \((X,A)\) and \((Y,B)\) be pairs of preordered spaces, and let \(f \colon X \to Y\) be a monotone map such that \(f(A) \subseteq B\). Then  
	the induced map \(f_*\colon H_*(X,A) \to H_*(Y,B)\) is a morphism of directional graded vector spaces.
\end{prop}

\begin{proof}
	Consider relative homology classes \(\alpha, \beta  \in H_*(X,A)\) such that ${\alpha \nearrow_{X,A} \beta}$. Then there exist subspaces $E, F \subseteq X$ such that 
	\begin{enumerate}[label=(\roman*)]
		\item $\alpha \in \im H_\ast ((E, E\cap A) \hookrightarrow (X,A))$;
		\item $\beta \in \im H_* ((F, F\cap A) \hookrightarrow (X,A))$;
		\item $\forall\, x\in E, y\in F \quad x\preceq_X y$.
	\end{enumerate} 
The commutative diagram  
\[
\begin{tikzcd}
	H_* (E, E \cap A) \arrow{r}{g_\ast} \arrow[swap]{d} & H_* (f(E), f(E) \cap B) \arrow{d} \\
	H_* (X,A) \arrow[swap]{r}{f_\ast} & H_* (Y, B),
\end{tikzcd}
\]
in which \(g_*\) is induced by the restriction of \(f\) and the vertical maps are induced by the inclusions, shows that
$$f_\ast (\alpha) \in \im H_* ((f(E), f(E)\cap B) \hookrightarrow (Y,B)).$$ In a similar way, we have that $f_\ast (\beta) \in \im H_* ((f(F), f(F)\cap B) \hookrightarrow (Y,B)).$

Consider now $a\in f(E), b \in f(F)$. Then there exist $x\in E$ and $y\in F$ such that $a=f(x)$ and $b=f(y)$. By (iii) and since $f$ is monotone map, it follows that $a\preceq_Y b$. We then conclude that $f_\ast (\alpha) \nearrow_{Y,B} f_\ast (\beta)$. Hence also $f_\ast (\alpha) \searrow_{Y,B} f_\ast (\beta)$. By Lemma \ref{1.5}, it follows that \(f_*\) is a morphism of directional graded vector spaces.
\end{proof}


\paragraph{Dihomotopy invariance.} Proposition \ref{nat} immediately implies that the homology digraph is a directed homotopy invariant. More precisely, we have the following result:

\begin{theor}
	Let \(f\colon X \to Y\) be a monotone map of preordered spaces that is a homotopy equivalence with a monotone homotopy inverse. Then  \(f_*\colon H_*(X) \to H_*(Y)\) is an isomorphism of directional graded vector spaces.
\end{theor}

\paragraph{Exact sequences.} By Proposition \ref{nat}, the homomorphisms induced by the inclusions in the long exact homology sequence of a pair of preordered spaces are the components of morphisms of directional graded vector spaces. By the following proposition, this also holds for the connecting homomorphisms:

\begin{prop}
	Let \((X, A)\) be a pair of preordered spaces. Then the connecting homomorphisms \(H_k(X,A) \to H_{k-1}(A)\) of the long exact homology sequence of  \((X,A)\) constitute a morphism of directional graded vector spaces. 
\end{prop}

\begin{proof}
The proof is an easy adaptation of that of Proposition \ref{nat}. 
\end{proof}

\paragraph{Excision.} We next show that the excision theorem for singular homology extends to the homology digraph. Given a subset \(A\) of a topological space \(X\), we denote by \(\overline{A}^X\) its closure and by \({\rm int}_X(A)\) its interior. 

\begin{theor}
	Let \((X,A)\) be a pair of preordered spaces, and let \(U\) be a subset of \(A\) such that \({\overline{U}^X \subseteq {\rm{int}}_X(A)}\). 
	Then the inclusion induces an isomorphism of directional graded vector spaces \(H_*(X\setminus U, A\setminus U) \rightarrow H_*(X,A)\). 
	
\end{theor}

\begin{proof}
	Let \(j\) be the inclusion \( {(X\setminus U, A\setminus U) \hookrightarrow (X,A)}\). Then, by the usual excision theorem for topological spaces, ${j_\ast \colon H_*(X\setminus U, A\setminus U) \to  H_*(X,A)}$ is an isomorphism of graded vector spaces. By Proposition \ref{nat}, it suffices to show that $j^{-1}_\ast $ is a morphism of directional graded vector spaces. Let ${\alpha, \beta  \in H_* (X\setminus U, A\setminus U)}$ be homology classes such that $j_\ast (\alpha) \nearrow_{X,A} j_\ast (\beta)$. By Lemma  \ref{1.5}, it is enough to show that \(\alpha \nearrow_{X\setminus U,A\setminus U} \beta\). Let $E, F \subseteq X$ be subspaces such that 
	\begin{itemize}
		\item[(i)] $j_\ast (\alpha) \in \im H_* ((E, E\cap A) \hookrightarrow (X,A))$;
		\item[(ii)] $j_\ast (\beta) \in \im H_* ((F, F\cap A) \hookrightarrow (X,A))$;
		\item[(iii)] $\forall\, x\in E, y\in F \quad x\preceq_X y$.
	\end{itemize}

	
	We have $E\cap U \subseteq E\cap A \subseteq E$ and  $$\overline{E\cap U}^E \subseteq E \cap \overline{U}^X  \subseteq E \cap {\rm{int}}_X (A)  \subseteq {\rm{int}}_E (E\cap A).$$ 
	Therefore, by excision, both horizontal maps in the following commutative diagram, in which all maps are induced by the inclusions, are isomorphisms:		
		\[
	\begin{tikzcd}
		H_* (E\setminus (E\cap U), (E \cap A)\setminus (E\cap U)) \arrow{r}{\cong} \arrow[swap]{d}{} & H_* (E, E\cap A) \arrow{d}{} \\
		H_* (X\setminus U,A\setminus U) \arrow[r,"\cong","j_\ast"'] & H_* (X,A)
	\end{tikzcd}
	\]
	This diagram shows that 
	\[\alpha \in \im H_*((E\setminus (E\cap U), (E \cap A)\setminus (E\cap U))\hookrightarrow (X\setminus U,A\setminus U)).\]
	Since $(E \cap A) \setminus (E \cap U) = (E \setminus U) \cap (A \setminus U)$, this means that
	\[\alpha \in \im H_*((E\setminus U, (E \setminus U) \cap (A \setminus U))\hookrightarrow (X\setminus U,A\setminus U)).\]	
	In a similar way, 
	$$\beta \in \im H_* ((F\setminus U, (F\setminus U) \cap (A\setminus U)) \hookrightarrow (X\setminus U, A\setminus U)).$$
	Finally, since $X\setminus U \subseteq X$, we have $x \preceq_{X\setminus U} y$ for all $x\in E\setminus U$ and $y\in F\setminus U$, and so we conclude that $\alpha \nearrow_{X\setminus U, A\setminus U} \beta$. 
\end{proof}

\paragraph{Coproducts.} The homology digraph is compatible with coproducts: 

\begin{theor} \label{sum}
	Let \((X_i)_{i\in \I}\) be a family of preordered spaces. 
	Then the inclusions \(\iota_j \colon X_j \hookrightarrow \coprod_{i\in \I}X_i\) induce an isomorphism of directional graded vector spaces
	\[\bigoplus_{i\in \I} H_*(X_i) \to H_*(\coprod_{i\in \I} X_i).\]
	
\end{theor}

\begin{proof}
	Since  \(\bigoplus_{i\in \I} H_*(X_i)\) is the coproduct of the directional graded vector spaces \(H_*(X_i)\), Proposition \ref{nat} implies that the isomorphism of graded vector spaces \(\varphi \colon \bigoplus_{i\in \I} H_*(X_i) \to H_*(\coprod_{i\in \I} X_i)\) induced by the inclusions \(\iota_j\) is a morphism of directional graded vector spaces. 
	
	In order to show that \(\varphi^{-1}\) is a morphism of directional graded vector spaces, consider nonzero classes 
	\({\xi, \zeta \in H_*(\coprod_{i\in \I}X_i)}\) such that \({\xi \nearrow_{\coprod_{i\in \I}X_i} \zeta}\), and let \(E, F \subseteq  \coprod_{i\in \I}X_i\) be subspaces such that
	\begin{enumerate}[label=(\roman*)]
		\item $\xi \in \im H_* (E \hookrightarrow \coprod_{i\in \I}X_i)$;
		\item $\zeta \in \im H_* (F \hookrightarrow \coprod_{i\in \I}X_i)$;
		\item $\forall\, x\in E, y\in F \quad x\preceq_{\coprod_{i\in \I}X_i} y$.
	\end{enumerate}
	Since \(\xi\not=0\) and \(\zeta\not=0\), we have \(E\not=\emptyset\) and \(F\not= \emptyset\). By (iii), it follows that there exists an index \(j\in \I\) such that \(E,F \subseteq X_j\). The commutative diagram
	\[
	\begin{tikzcd}
		H_* (E) \arrow{r}{=} \arrow[swap]{d}{} & H_* (E) \arrow{d}{} \\
		H_* (X_j) \arrow{r}{\iota_{j\ast}} & H_* (\coprod_{i\in \I}X_i),
	\end{tikzcd}
	\] 
	in which the vertical maps are induced by the inclusions, shows that there exists a homology class \(\upsilon \in \im H_*(E \hookrightarrow X_j)\) such that \(\iota_{j\ast}(\upsilon) = \xi\). Similarly, there exists a homology class \(\nu \in \im H_*(F \hookrightarrow X_j)\) such that \(\iota_{j\ast}(\nu) = \zeta\). By (iii), \(\upsilon \nearrow_{X_j} \nu\) and hence \(\upsilon \searrow_{X_j} \nu\). Viewing \(\upsilon\) and \(\nu\) as elements of \(\bigoplus_{i\in \I} H_*(X_i)\), we obtain that \(\varphi^{-1}(\xi) = \upsilon \searrow_{\bigoplus_{i\in \I} H_*(X_i)} \nu = \varphi^{-1}(\zeta)\). By Lemma \ref{1.5}, it follows that \(\varphi^{-1}\) is a morphism of directional graded vector spaces. 	
\end{proof}

\begin{rem}
	Let \(X\) and \(Y\) be preordered spaces, and let \(x_0 \in X\) and \({y_0\in Y}\) be minimal elements such that the inclusions  \(\{x_0\} \hookrightarrow X\) and \({\{y_0\} \hookrightarrow Y}\) are closed cofibrations. Then the wedge
	\[X\vee Y = \faktor{X\amalg Y}{x_0 \sim y_0}\]
	is a preordered space with respect to the preorder given by
	\[a \preceq_{X\vee Y} b  \iff  a, b \in X, a \preceq_X b \;\, \mbox{or} \;\, a, b \in Y, a \preceq_Y b,\]    
	and the proof of Theorem \ref{sum} can be adapted to show that the inclusions \((X,x_0) \hookrightarrow (X \vee Y,x_0=y_0)\) and \((Y,y_0) \hookrightarrow (X \vee Y,x_0=y_0)\) induce an isomorphism of directional graded vector spaces
	\[H_*(X,x_0) \oplus H_*(Y,y_0) \to H_*(X\vee Y,  x_0 = y_0).\]
 
    It is also possible to generalize the arguments of the proof of Theorem \ref{sum} to establish that given a family \(((X_i, A_i))_{i\in \I}\) of pairs of preordered spaces, the inclusions \((X_j, A_j) \hookrightarrow \coprod_{i\in \I}(X_i, A_i)\) induce an isomorphism of directional graded vector spaces
	\[\bigoplus_{i\in \I} H_*(X_i, A_i) \to H_*(\coprod_{i\in \I} X_i, \coprod_{i\in \I}A_i).\]		
\end{rem}


\section{The homology digraph of a product}

Let $X$ and $Y$ be preordered spaces. By the topological K\"unneth theorem, the homology cross product
\[\times \colon H_*(X)\otimes H_*(Y) \to H_*(X\times Y),\; \alpha\otimes \beta \mapsto \alpha \times \beta\]
is an isomorphism of graded vector spaces. In this section, we prove that it actually is an isomorphism of directional graded vector spaces. It is thus possible to compute the homology digraph of \(X\times Y\) from the homology digraphs of \(X\) and \(Y\). 

\begin{lem} \label{crossprop}
	Let $\alpha, \alpha' \in H_*(X)$ and ${\beta, \beta' \in H_*(Y)}$ be homology classes such that $\alpha \nearrow_X \alpha'$ and $\beta \nearrow_Y \beta'$. Then $\alpha \times \beta \nearrow_{X\times Y} \alpha'\times \beta'$. 
\end{lem}

\begin{proof}
	Consider subspaces $E$, $E'\subseteq X$ and $F$, $F' \subseteq Y$ such that ${\alpha \in \im\, H_*(E \hookrightarrow X)}$, ${\alpha' \in \im\, H_*(E' \hookrightarrow X)}$, ${\beta \in \im\, H_*(F \hookrightarrow Y)}$, ${\beta' \in \im\, H_*(F' \hookrightarrow Y)}$, $x \preceq_X x'$ for all points $x \in E$ and $x' \in E'$, and $y \preceq_Y y'$ for all points $y \in F$ and $y' \in F'$. The commutative diagram
	\[
	\xymatrix{
		H_*(E) \otimes H_*(F) \ar[r]^(.55){\times} \ar[d]_{} & H_*(E\times F) \ar[d]^{}\\
		H_*(X) \otimes H_*(Y) \ar[r]_(.55){\times}  & H_*(X\times Y),
	}
	\]
	in which the vertical maps are induced by the inclusions, 
	shows that ${\alpha \times \beta \in \im \, H_*(E \times F \hookrightarrow X \times Y)}$. An analogous argument shows that ${\alpha' \times \beta' \in {\im \, H_*(E' \times F' \hookrightarrow X \times Y)}}$. Since $(x,y) \preceq_{X\times Y} (x',y')$ for all points ${(x,y) \in E\times F}$ and $(x',y')\in E'\times F'$, we have $\alpha \times \beta \nearrow_{X\times Y} \alpha'\times \beta'$. 
\end{proof}

\begin{lem} \label{splittimes}
	Consider $\xi,\xi' \in H_*(X\times Y)$  such that ${\xi \nearrow_{X\times Y} \xi'}$. Then there exist homology classes $\alpha_i, \alpha'_j \in H_*(X)$  and $\beta_i,  \beta'_j \in H_*(Y)$ $(i = 1, \dots, m,$ $j= 1, \dots, n)$ such that $\xi = \sum_i \alpha_i \times \beta_i$, $\xi' = \sum_j \alpha'_j \times \beta'_j$, and $\alpha_i \nearrow_X \alpha'_j$ and $\beta_i \nearrow_Y \beta'_j$ for all $i$ and $j$.
\end{lem}

\begin{proof}
    Let $U$ and $U'$ be subspaces of ${X \times Y}$ such that 
    \begin{enumerate}[label=(\roman*)]
        \item $\xi \in {\im \, H_*(U \hookrightarrow X\times Y)}$;
        \item $\xi' \in \im \, {H_*(U' \hookrightarrow X\times Y)}$;
        \item \(\forall \, (x,y) \in U, (x',y') \in U'\quad (x,y) \preceq_{X\times Y} (x',y')\).
    \end{enumerate}
    Consider the projections \(pr_X\colon X\times Y \to X\) and \(pr_Y\colon X \times Y\to Y\), and write \(E = pr_X(U)\), \(F = pr_Y(U)\), \(E' = pr_X(U')\), and \(F' = pr_Y(U')\). Then 
    \begin{enumerate}[label=(\arabic*)]
        \item $U \subseteq E\times F$;
        \item $U' \subseteq E'\times F'$;
        \item \(\forall\, x \in E, x' \in E'\quad  x \preceq_X x'\); 
        \item \(\forall\, y \in F, y'\in F'\quad y \preceq_Y y'\). 		
    \end{enumerate}
    By conditions (i), (ii), (1), and (2), we have $\xi \in \im H_*(E\times F \hookrightarrow X\times Y)$ and $\xi' \in \im H_*(E'\times F' \hookrightarrow X\times Y)$. The commutative diagram
    \[
    \xymatrix{
        H_*(E) \otimes H_*(F) \ar[r]^(.55){\cong}_(.55){\times} \ar[d] & H_*(E\times F) \ar[d]\\
        H_*(X) \otimes H_*(Y) \ar[r]^(.55){\cong}_(.55){\times}  & H_*(X\times Y),
    }
    \]
    in which the vertical maps are induced by the inclusions, shows that there exist homology classes $\alpha_i \in \im H_*(E \hookrightarrow X)$  and $\beta_i \in \im H_*(F\hookrightarrow Y)$ ${(i = 1, \dots, m)}$ such that $\xi = \sum_i \alpha_i \times \beta_i$. Similarly, there exist homology classes $\alpha'_j \in \im H_*(E' \hookrightarrow X)$  and $\beta'_j \in \im H_*(F'\hookrightarrow Y)$ $(j = 1, \dots, n)$ such that $\xi' = \sum_j \alpha'_j \times \beta'_j$. By condition (3), $\alpha_i \nearrow_X \alpha'_j$ for all $i$ and $j$. By condition (4), $\beta_i \nearrow_Y \beta'_j$ for all $i$ and $j$. 
\end{proof}

\begin{theor}\label{cross}
	The cross product \(\times \colon H_*(X)\otimes H_*(Y) \to H_*(X\times Y)\) is an isomorphism of directional graded vector spaces.
\end{theor}

\begin{proof}
	Let $\theta, \theta' \in H_*(X)\otimes H_*(Y)$, and suppose that $\theta \searrow_{H_*(X)\otimes H_*(Y)} \theta'$. Then $\theta \otimes \theta' \in R^{\searrow_{H_\ast(X)\otimes H_\ast(Y)}}$ and so, by Proposition \ref{tensor}, $$\theta \otimes \theta' = \sum_i \lambda_i v_i \otimes w_i \otimes v'_i \otimes w'_i, \quad \lambda_i \in \K, \medspace v_i \nearrow_X v'_i, \medspace w_i \nearrow_Y w'_i.$$ We want to prove that $\times (\theta) \searrow_{X\times Y} \times (\theta')$, i.e., $\times (\theta) \otimes \times (\theta') \in R^{\searrow_{X\times Y}}$. We have that 
	\[\times (\theta) \otimes \times (\theta') = \sum_i \lambda_i  ( v_i \times w_i) \otimes (v'_i \times w'_i)\]
	with $v_i \times w_i \nearrow_{X\times Y} v'_i \times w'_i$, which follows from Lemma \ref{crossprop}. Hence ${\times (\theta) \otimes \times (\theta') \in R^{\searrow_{X\times Y}}}$.

	Suppose now that $\times (\theta) \searrow_{X\times Y} \times (\theta')$. Then, by Proposition \ref{1.3}, $$\times (\theta ) \otimes \times (\theta') \in \langle \xi\otimes \xi' \mid  \xi\nearrow_{X\times Y} \xi' \rangle,$$ i.e., $$\times (\theta) \otimes \times(\theta') = \sum_r \lambda_r \xi_r \otimes \xi'_r,$$ with \(\lambda_r \in \K\) and $\xi_r \nearrow_{X\times Y} \xi'_r$. By Lemma \ref{splittimes}, there exist
	\[\alpha_{r,i_r}, \alpha'_{r,j_r} \in H_\ast (X), \;\beta_{r,i_r}, \beta'_{r,j_r} \in H_\ast (Y)  \quad (i_r = 1, \dots, m_r,\, j_r = 1, \dots, n_r)\]
	such that	
	$$\xi_r = \sum_{i_r} \alpha_{r, i_r} \times \beta_{r,i_r}, \quad \xi'_r = \sum_{j_r} \alpha'_{r, j_r} \times \beta'_{r,j_r},$$
and	
	$$ \forall \, i_r, j_r\quad \alpha_{r,i_r} \nearrow_X \alpha'_{r,j_r}, \; \beta_{r,i_r} \nearrow_Y \beta'_{r,j_r}.$$ 	
	So, 	
	\begin{eqnarray*}
		\times \otimes \times (\theta \otimes \theta' ) & = & \times (\theta) \otimes \times (\theta') \\
		& = & \sum_r \lambda_r \xi_r \otimes \xi'_r \\
		& = & \sum_r \lambda_r  \sum_{i_r, j_r} \big( \alpha_{r,i_r} \times \beta_{r,i_r}  \big) \otimes   \big( \alpha'_{r,j_r} \times \beta'_{r,j_r}  \big)  \\
		& = & \times \otimes \times \Big( \sum_r \sum_{i_r, j_r} \lambda_r  (\alpha_{r,i_r} \otimes \beta_{r,i_r}) \otimes (\alpha'_{r,j_r} \otimes \beta'_{r,j_r})  \Big).
	\end{eqnarray*}	
	Hence, by injectivity of $\times \otimes \times$, $$\theta \otimes \theta' = \sum_r \sum_{i_r, j_r}  \lambda_r \alpha_{r,i_r} \otimes \beta_{r,i_r} \otimes \alpha'_{r,j_r} \otimes \beta'_{r,j_r}.$$ Thus, $\theta \searrow_{H_*(X)\otimes H_*(Y)} \theta'$.	 
\end{proof}

\begin{ex}
	In this example, we will calculate the homology digraph of the torus \(\vec S^1\times O^1\) from the homology digraphs of \(\vec S^1\) and \(O^1\).
    We suppose that $\K = \Z_2$. Let $\alpha \in H_0 (\vec S^1)$, $\beta \in H_1 (\vec S^1)$ and $\gamma \in H_0(O^1)$, $\delta \in H_1 (O^1)$ be the generators. By Examples \ref{hdorddircirc}(i) and \ref{hdorddircirc}(iii), we have that $\alpha \searrow_{\vec S^1} \alpha$, $\alpha \searrow_{\vec S^1} \beta$, $ \beta \searrow_{\vec S^1} \alpha$, $\beta \searrow_{\vec S^1} \beta$ and that $\gamma \searrow_{O^1} \gamma$, $\gamma \searrow_{O^1} \delta$, $\delta \searrow_{O^1} \gamma$, $\delta \centernot{\searrow}_{O^1}\delta$. The defining vector space of $\searrow_{H_\ast (\vec S^1) \otimes H_\ast (O^1)}$ is 
\begin{align*}
	R & =  \langle v\otimes w \otimes v' \otimes w' \mid v\searrow_{\vec{S}^1} v', w\searrow_{O^1} w'   \rangle \\  
	& =  \langle  \alpha \otimes \gamma \otimes \alpha \otimes \gamma; \\  
	 & \quad \medspace \; \,   \alpha \otimes \gamma \otimes \beta \otimes \gamma, \beta \otimes \gamma \otimes \alpha \otimes \gamma, \alpha \otimes \gamma \otimes \alpha \otimes \delta, \alpha \otimes \delta \otimes \alpha \otimes \gamma; \\
	  & \quad \medspace \; \, \alpha \otimes \gamma \otimes \beta \otimes \delta, \beta \otimes \gamma \otimes \alpha \otimes \delta,  \alpha \otimes \delta \otimes \beta \otimes \gamma, \beta \otimes \delta \otimes \alpha \otimes \gamma, \\
	 & \quad \medspace \; \quad \beta \otimes \gamma \otimes \beta \otimes \gamma; \\
	 & \quad \medspace \; \,   \beta \otimes \gamma \otimes \beta \otimes \delta, \beta \otimes \delta \otimes \beta \otimes \gamma \rangle.
\end{align*}
The nonzero elements of $H_\ast (\vec S^1) \otimes H_\ast (O^1)$ are 
$\alpha \otimes \gamma$, $\alpha \otimes \delta$, $\beta \otimes \gamma$, $\alpha \otimes \delta + \beta \otimes \gamma$, and $\beta \otimes \delta$.
By inspection,  
$$\alpha \otimes \gamma, \beta \otimes \gamma \searrow_{H_\ast (\vec S^1) \otimes H_\ast (O^1)} \alpha \otimes \gamma, \alpha \otimes \delta, \beta \otimes \gamma, \alpha \otimes \delta + \beta \otimes \gamma, \beta \otimes \delta;$$
$$\alpha \otimes \delta, \alpha \otimes \delta + \beta \otimes \gamma, \beta \otimes \delta \searrow_{H_\ast (\vec S^1) \otimes H_\ast (O^1)} \alpha \otimes \gamma, \beta \otimes \gamma;$$
$$\alpha \otimes \delta, \alpha \otimes \delta + \beta \otimes \gamma, \beta \otimes \delta \not \searrow_{H_\ast (\vec S^1)\otimes H_\ast (O^1)} \alpha \otimes \delta, \alpha \otimes \delta + \beta \otimes \gamma, \beta \otimes \delta.$$
The nonzero homology classes of the torus are $\alpha\times \gamma$, $\alpha\times \delta$, $\beta\times \gamma$, $\alpha\times \delta + \beta \times \gamma$, and $\beta \times \delta$. By Theorem \ref{cross}, the cross products $\alpha \times \gamma$ and $\beta\times \gamma$ point to all nonzero homology classes and the only nonzero homology classes pointed to by  $\alpha\times \delta$,  $\alpha\times \delta + \beta \times \gamma$, and $\beta \times \delta$ are $\alpha\times \gamma$ and $\beta \times \gamma$.
\end{ex}

\begin{rem}
    The arguments of this section can easily be generalized to establish the following relative version of Theorem \ref{cross}: if \(A\) is open in \(X\) and \(B\) is open in \(Y\), then the relative cross product \[\times \colon H_*(X,A)\otimes H_*(Y,B) \to H_*(X\times Y, A\times Y\cup X \times B),\; \alpha\otimes \beta \mapsto \alpha \times \beta\]
    is an isomorphism of directional graded vector spaces.     
\end{rem}

\end{sloppypar}

\bibliography{refs}

\begin{thebibliography}{10}

\bibitem{DubutGG}
J.~Dubut, E.~Goubault, and J.~Goubault-Larrecq.
\newblock {Directed Homology Theories and Eilenberg-Steenrod Axioms}.
\newblock {\em Applied Categorical Structures}, 25(5):775--807, 2017.

\bibitem{FahrenbergDiH}
U.~Fahrenberg.
\newblock {Directed Homology}.
\newblock {\em Electronic Notes in Theoretical Computer Science}, 100:111--125,
  2004.

\bibitem{FGHMR}
L.~Fajstrup, E.~Goubault, E.~Haucourt, S.~Mimram, and M.~Raussen.
\newblock {\em {Directed Algebraic Topology and Concurrency}}.
\newblock Springer, 2016.

\bibitem{FajstrupGR}
L.~Fajstrup, M.~Rau{\ss}en, and E.~Goubault.
\newblock {Algebraic topology and concurrency}.
\newblock {\em Theoret. Comput. Sci.}, 357:241--278, 2006.

\bibitem{CatarinaMestrado}
C.~Faustino.
\newblock Homologia dirigida.
\newblock Master's thesis, Universidade do Minho, Braga, Portugal, 2021.

\bibitem{Goubault}
E.~Goubault.
\newblock {Some geometric perspectives in concurrency theory}.
\newblock {\em Homology, Homotopy and Applications}, 5(2):95--136, 2003.

\bibitem{GoubaultJensen}
E.~Goubault and T.P. Jensen.
\newblock {Homology of higher-dimensional automata}.
\newblock In {\em {Proc. of CONCUR '92}}, volume 630 of {\em Lecture Notes in
  Computer Science}, pages 254--268. Springer, 1992.

\bibitem{GrandisDirHoTI}
M.~Grandis.
\newblock Directed homotopy theory, {I}.
\newblock {\em Cahiers de Topologie et G\'eom\'etrie Diff\'erentielle
  Cat\'egoriques}, 44(4):281--316, 2003.

\bibitem{GrandisDiH}
M.~Grandis.
\newblock {Directed combinatorial homology and noncommutative tori (the
  breaking of symmetries in algebraic topology)}.
\newblock {\em Math. Proc. Cambridge Philos. Soc.}, 138(2):233--262, 2005.

\bibitem{GrandisBook}
M.~Grandis.
\newblock {\em {Directed Algebraic Topology: Models of Non-Reversible Worlds}},
  volume~13 of {\em New Mathematical Monographs}.
\newblock Cambridge University Press, 2009.

\bibitem{reldi}
T.~Kahl.
\newblock {Relative directed homotopy theory of partially ordered spaces}.
\newblock {\em Journal of Homotopy and Related Structures}, 1(1):79--100, 2006.

\bibitem{hgraph}
T.~Kahl.
\newblock {The homology graph of a precubical set}.
\newblock {\em Homology, Homotopy and Applications}, 16(1):119--138, 2014.

\bibitem{KrishnanConvCat}
S.~Krishnan.
\newblock {A Convenient Category of Locally Preordered Spaces}.
\newblock {\em Applied Categorical Structures}, 17:445--466, 2009.

\bibitem{KrishnanNorth}
S.~Krishnan and P.R. North.
\newblock A {H}urewicz model structure for directed topology.
\newblock {\em Theory Appl. Categ.}, 37:Paper No. 20, 613--634, 2021.

\end{thebibliography}
\bibliographystyle{plain}

\end{document}